\documentclass[12pt]{amsart}
\usepackage{amssymb}
\usepackage{amsfonts}
\usepackage{graphicx}
\usepackage{amsmath}
\usepackage{xcolor}

\newtheorem{theorem}{Theorem}

\newtheorem{corollary}[theorem]{Corollary}

\newtheorem{lemma}[theorem]{Lemma}

\newtheorem{proposition}[theorem]{Proposition}

\newcommand{\C}{\mathbb{C}}
\newcommand{\N}{\mathbb{N}}

\newcommand{\R}{\mathbb{R}}

\begin{document}

\title{Asympotic bounds for Bombieri's inequality on  products of
homogeneous polynomials}
\author{ J. M. Aldaz and H. Render}
\address{H. Render: School of Mathematical Sciences, University College
Dublin, Dublin 4, Ireland.}
\email{hermann.render@ucd.ie}
\address{J.M. Aldaz: Instituto de Ciencias Matem\'aticas (CSIC-UAM-UC3M-UCM)
and Departamento de Matem\'aticas, Universidad Aut\'onoma de Madrid,
Cantoblanco 28049, Madrid, Spain.}
\email{jesus.munarriz@uam.es}
\email{jesus.munarriz@icmat.es}
\thanks{The first named author was partially supported by Grant PID2019-106870GB-I00 of the
MICINN of Spain,  by ICMAT Severo Ochoa project 
CEX2019-000904-S (MICINN), and by  V PRICIT (Comunidad de Madrid - Spain).}

\maketitle



\markboth{J. M. Aldaz and H. Render}{Asympotic bounds for Bombieri's inequality}

\begin{abstract} Let $P$ be a fixed homogeneous polynomial. We present a sharp condition on $P$ guaranteeing  the existence of asymptotically larger bounds in Bombieri's inequality, so for every homogeneous polynomial $q_m$ of degree $m$ we have
\begin{equation*}
\left\Vert P q_{m}\right\Vert _{a}\geq C_{P} m^{l\left(  P\right)
/2}\left\Vert q_{m}\right\Vert _{a}, 
\end{equation*}
where $\| \cdot \| _{a}$ denotes
 the apolar norm. Explicit estimates for $C_P > 0$ and $l(P) > 0$ are given. 
\end{abstract}

\section{Introduction}

Bombieri's inequality for the apolar norm $\|\cdot\|_a$ (defined using the apolar inner product, cf. formula (\ref{eqQPD}) below) states that given two homogeneous polynomials $f$ and $g$, their product satisfies the bound
$$
\|f g\|_a \ge \|f\|_a \|g\|_a.
$$
 The inequality, which originally appeared in \cite{BBEM90} using a different normalization (cf. Section 2) is sharp: equality holds precisely  when the polynomials depend on disjoint sets of coordinates, perhaps after applying a unitary change of variables, cf. \cite[Main Theorem]{Rezn93}, cf. also \cite[Theorem 3]{BeFrMi96}.

For some applications related to finding entire (and thus global) solutions to boundary value problems for PDE's, when the boundary is an algebraic variety and the data are given by an entire function (see for instance 
\cite{AlRe23}, 
\cite{EbSh95}, 
\cite{EbSh96}, 
\cite{KhSh92}) asymptoticly stronger variants   of Bombieri's inequality are required. Let $P_k$ be a fixed homogeneous polynomial of degree $k$. The aforementioned  variants of Bombieri's inequality take the form 
\begin{equation}
\left\Vert P_k q_{m}\right\Vert _{a}\geq C_{P_k}\cdot m^{l\left(  P_k\right)
/2}\cdot\left\Vert q_{m}\right\Vert _{a}, \label{eqmain}
\end{equation}
where  $q_{m}$  ranges over all the 
homogeneous polynomials of degree $m$, while $C_{P_k}$ and $l\left(  P_k\right)$ are strictly positive and  depend only on
 $P_k$. We call the largest non-negative
integer $l\left(  P_k \right)  \ge 0$ making inequality (\ref{eqmain}) true  the {\em Khavinson-Shapiro exponent}. If $l\left(  P_k \right)  > 0 $, we say that $P_k$ satisfies the {\em Khavinson-Shapiro bounds} with exponent $l\left(  P_k\right)$; this terminology is motivated by the fact that such bounds,  for the exponent $l(P_k) = 1$, were first considered in \cite{KhSh92}.
The case $l(P_k) = 0$ corresponds to Bombieri's inequality, with $C_{P_k} = \|P_k\|_a$.

\vskip .2cm

Let us denote by $\mathcal{P}\left(  \mathbb{C}^{d}\right)  = \C[z_1, \dots , z_d]$ the set of all
polynomials  in  $z=\left(  z_{1},\dots ,z_{d}\right)  \in
\mathbb{C}^{d}$ with complex coefficients, by $\mathcal{P}_{m}\left(
\mathbb{C}^{d}\right)  $ the subspace of all homogeneous polynomials
of degree $m$,  by  $P^{\ast}$  the
polynomial obtained from $P$ by conjugating its
coefficients, and by $\left\langle \cdot ,\cdot \right\rangle_2$ the canonical hermitian inner product on $\C^d$.  The next theorem presents a necessary and sufficient condition which establishes when the Khavinson-Shapiro bounds hold for a given exponent $l(P_k)$.

\begin{theorem}
\label{ThmMain1} Let $P_k: \C^d \to \C$ be a non-zero homogeneous polynomial of degree $k \ge 2$, and let  $1 \le \rho < k$  be a natural number.
If the polynomials contained in
$$
\{\partial^{\alpha}P_k: \left\vert \alpha\right\vert =\rho\}
$$ 
do not have a
common zero $c\in \C^d \setminus \{0\},$ then taking 
$$
C_1 := 
k^{ - \rho} \binom{d+\rho-1}{\rho}^{-1/2} \min_{\eta\in\mathbb{S}^{2d-1}}
\left({\displaystyle\sum_{\left\vert \alpha\right\vert =\rho}}
\left\vert \partial^{\alpha}P\left(  \eta\right)  \right\vert ^{2}\right)^{1/2}
$$
we have  that for all $m\geq
0$ and all $f_{m}\in\mathcal{P}_{m}\left(  \mathbb{C}^{d}\right)$, 
\begin{equation}\label{KS}
\left\Vert P_k\cdot f_{m}\right\Vert _{a}\geq C_1\cdot (m +d)^{(k-\rho)/2}
\cdot\left\Vert f_{m}\right\Vert _{a}.
\end{equation}
On the other hand, if there is a 
$c\in\mathbb{C}^{d}$ with $|c| = 1$ and
$
\left(  \partial^{\alpha}P_k\right)  \left( c \right)  =0$  for all
 multi-indices $\alpha$ such that $\left\vert \alpha\right\vert =\rho,$
then taking 
$$
C_2 =
\left(
{\displaystyle\sum_{\rho<\left\vert \alpha\right\vert \leq k}}
\frac{\left\vert \left(  \partial^{\alpha}P^{\ast}\right)  \left(   {c}\right)
\right\vert ^{2}}{\alpha!}
\right)^{1/2},
$$
 we obtain that   for
all $m\geq k$, 
\begin{equation}\label{KSupper}
\left\Vert P_k\left(  z \right)  \left\langle z,    \overline{c} \right\rangle_2^{m}\right\Vert
_{a}
\leq
 C_2\cdot m^{(k-\rho -1)/2}\cdot\left\Vert \left\langle z , \overline{c} \right\rangle_2
^{m}\right\Vert _{a}.\
\end{equation} \end{theorem}
 By homogeneity, assuming in the preceding result that there is 
a common root $c$ with $|c| = 1$, is equivalent to  the seemingly weaker assumption that a common root $c \ne 0$ exists. 

The theorem is stated for $k\ge 2$. The case 
$k= 1$ was solved in \cite[Theorem 5]{AlRe23}:
 if $d > 1$ and $k \ge 1$, then every
homogeneous polynomial $P_k$ of degree $k$ satisfies $l(P_k) \le k-1$. In particular, if $d > 1$ and $k= 1$ then   $l(P_k) = 0$, and thus no homogeneous polynomial $P_1$ satisfies the Khavinson-Shapiro bounds.  On the other hand, if $d= 1$ and $k=1$, then we always have  $l(P_1) = 1$, again by \cite[Theorem 5]{AlRe23}. 
   
 \vskip .2 cm
   
Theorem \ref{ThmMain1} furnishes a computable criterion for finding out whether Khavinson-Shapiro bounds hold and with which exponent $l(P_k)$, since we only need
to obtain, for some $\rho = 0, 1,\dots , k-1$, a  polynomial system of the form $
\{\partial^{\alpha}P_k: \left\vert \alpha\right\vert =\rho\}
$ without non-zero solutions, and this can always be determined via Buchberger's algorithm, by writing a Gr\"obner basis for the above set and finding whether non-zero solutions exist. Thus, the necessary and sufficient criterion from Theorem \ref{ThmMain1} can be checked in a finite number of steps, for any given $P_k$. 

Let us illustrate this with a couple of examples. First take $P_4\left(  x , y \right)  =  x^{4}+ 4  x^2 y^2 + 2 y^{4}$.  Then $\partial_1  P_4\left(  x , y \right)  = 4 x^3 + 8  x  y^2$ and $\partial_2  P_4\left(  x , y \right)  =  8   x^2 y + 8 y^3$. An online Gr\"obner bases calculator used   on $\{\partial_1  P_4, \partial_2  P_4\}$ yielded the Gr\"obner basis $\{y^5, x y^3, x^2 y + y^3, x^3 + 2 x y^2\}$, and $x = 0$, $y = 0$ as its (obviously) only common root. Thus, the Khavinson-Shapiro exponent in this case is $l(P) = 4 -1 = 3$. 

For a second example (where the answer was already known to be $l(P_4) = 2$, cf. \cite[Proof of Proposition 11]{EbRe08}) consider the polynomial 
 $P\left(  x , y \right)  =  \left\vert (x, y) \right\vert
^{4} = x^{4}+ 2 x^2 y^2 + y^{4}$.  Now $\partial_1  P\left(  x , y \right)  = 4 x^3 + 4  x  y^2$ and $\partial_2  P\left(  x , y \right)  =  4  x^2 y + 4 y^3$. 
 An  online Gr\"obner bases calculator for this new set $\{\partial_1  P, \partial_2  P\}$ yields the basis $\{x^2 y + y^3, x^3 + x y^2\}$ and  roots $y = i x$, $y = - ix$, so non-zero roots do exist. It follows from Theorem \ref{ThmMain1} that $l(P) \le 4 - 1 - 1 = 2$.
 This example shows that even if all the coefficients of the polynomial are real,   complex zeros need to be considered in order to calculate $l(P)$, for the only common real root of $\{\partial_1  P, \partial_2  P\}$ is $(0,0)$.
 
 Note next that  $\partial_{11}  P\left(  x , y \right)  = 12 x^2 + 8 y^2$, $\partial_{12}  P\left(  x , y \right)  = 8 x y$, and  $\partial_{22}  P\left(  x , y \right)  = 8 x^2 + 12 y^2$.
 An  online Gr\"obner bases calculator for  $\{\partial_{11}  P, \partial_{12}  P, \partial_{22}  P\}$ yields the basis $\{y^2, x y, x^2\}$, so the only common root is $(0,0)$. Thus  by Theorem \ref{ThmMain1} the Khavinson-Shapiro exponent in this instance  is $l(P) \ge 4 -2 = 2$. 
 
 \vskip .2 cm
 
Additionally, from  Theorem \ref{ThmMain1} we see that the case $l(P_k) \ge 1$  is
``generic'' for $k \ge 2$, in the same sense  that a generic set of $ \binom{d+ k-2}{k - 1}$  vectors in $\C^d$ will span the whole space, rather than all lie in a proper subspace. To see the relationship between both statements, 
suppose $P_k$ is a homogeneous polynomial of degree $k$ in $d$
variables. Then the Khavinson-Shapiro bounds hold if and only if the system of $ \binom{d+ k-2}{k - 1}$ linear functions $
\{\partial^{\alpha}P_k: \left\vert \alpha\right\vert = k - 1\}
$ does not have  a common nontrivial zero $c$. But the $k-1$-th derivatives of $P_k$ are linear polynomials, which can be identified in the usual way with vectors in $\C^d$. Thus, the existence of a  non-zero root $c$ of all the linear polynomials, is equivalent to saying that the vector $c$ is orthogonal to the subspace generated by the  vectors defining the said linear polynomials,  which in turn means that the dimension of that subspace is strictly less than $d$. 
 
 \vskip .2 cm
 
The notion of amenable polynomials was used in  \cite[Lemma 11]{KhSh92} to obtain Khavinson-Shapiro bounds with exponent $l(P) = 1$, making Lemma 11 of  \cite{KhSh92}  a precursor of Theorem \ref{ThmMain1}. In \cite[p. 464]{KhSh92},   a homogeneous
	polynomial $P_{k} \in \C[z_1, \dots , z_d]$ of degree $k > 0$ is said to be {\em amenable} if for each 
$j\in  \{1,\dots, d\}$ there is a multi-index $\alpha$, with $|\alpha | = k -1$, such that $D^\alpha P_k$ is a non-zero constant times $z_j$, say, $a_j z_j$. Obviously this entails that zero is the only common root for the set of $k-1$-derivatives of $P_k$.
Representing the linear function $D^\alpha P_k = a_jz_j$ with respect to the standard basis in $\C^d$, amenability means that by taking all the derivatives of order $k-1$ we obtain the basis $\{a_1 e_1, \dots , a_d e_d \}$ of $\C^d$.    Clearly this condition  is far more restrictive  than just requiring that the set of linear functions $
\{\partial^{\alpha}P_k: \left\vert \alpha\right\vert = k - 1\}
$ have rank $d$, as is done in Theorem \ref{ThmMain1}. 

\vskip .2 cm

Regarding
applications of Theorem \ref{ThmMain1}, we mention first the question of existence of entire solutions to
boundary value problems. In results such as \cite[Theorem 3.1.1]{EbSh96} or \cite[Theorem 2]{AlRe23}, the Khavinson-Shapiro bounds with a given exponent are simply {\em assumed}, so the problem  of determining whether these hold for any particular polynomial is not dealt with. Theorem \ref{ThmMain1} solves this in a computable fashion.

Bombieri's inequality has been applied to the factorization of polynomials, cf. for instance \cite{Beau92}. While this subject falls outside our area of expertise, it seems likely that the asymptotic sharpenings presented here have some relevance to this topic.

Finally, we mention another problem for which Theorem \ref{ThmMain1} yields new information.
Given a fixed polynomial $P$, consider, for different polynomial norms $\|\cdot\|$, the following quantities:
\begin{equation}
	I_{m}\left(  P\right)  =\inf\left\{  \| Pf_{m}\| : f_{m}
	\in\mathcal{P}_{m}\left(  \mathbb{C}^{d}\right)  \text{ and }\|
	f_{m}\|=1\right\}  \label{LowerAsymB}
\end{equation}
and 
\begin{equation}
	S_{m}\left(  P\right)  =\sup\left\{  \| Pf_{m}\| : f_{m}
	\in\mathcal{P}_{m}\left(  \mathbb{C}^{d}\right)  \text{ and }\|
	f_{m}\|=1\right\}  \label{upperAsymB},
\end{equation}
cf. for instance, \cite{Pina12}, \cite{Beau97}, and the references contained therein.
When working with the apolar norm we use the notations
$I_{m}^{a}\left(  P\right)  $ and $S_{m}^{a}\left(  P\right)  $.  In \cite[pg. 1071]{Rezn93}  a direct calculation for the
monomial $z^{\beta}$, with $\beta=\left(  \beta_{1},\dots ,\beta_{d}\right)
\in\mathbb{N}^{d}$ and  $\beta_{1}\geq\beta_{2}\geq....\geq\beta_{d}$,
leads to the following result:
\[
I_{m}^{a}\left(  z^{\beta}/\sqrt{\beta!}\right)  =\left\Vert \frac{z^{\beta}
}{\sqrt{\beta!}}\frac{z_{d}^{m}}{\left\Vert z_{d}^{m}\right\Vert _{a}
}\right\Vert _{a}
=
\sqrt{\frac{\left(  \beta_{d}+m\right)  !}{\beta_{d}!m!}}.
\]
B. Beauzamy  shows in   \cite{Beau97} that the optimimization problems (\ref{LowerAsymB}) and (\ref{upperAsymB}) for the
apolar norm, can be solved by finding the smallest and largest eigenvalues of a
certain matrix problem.  We also
mention that D. Pinasco has proven in \cite[Theorem 3.3]{Pina12} that for a homogeneous polynomial $P_k$ of
degree $k$,
\[
\lim\sup_{m\rightarrow\infty}\frac{S_{m}^{a}\left(  P_k\right)  }{\sqrt{\left(
		m+1\right)  \cdots\left(  m+k\right)  }}=\sup_{z\in\mathbb{C}^{d},\left\vert
	z\right\vert =1}\left\vert P_k\left(  z\right)  \right\vert .
\]
Beyond this, not much  seems to be  known about $I_{m}^{a}\left(  P\right)  $ and $S_{m}^{a}\left(  P\right)  $.
We note that Theorem \ref{ThmMain1}  provides respectively lower and upper estimates for $I^a_{m}\left(  P_k\right)  $ and $S^a_{m}\left(  P_k\right)  $.

\begin{corollary}
	\label{cor1} Let $P_k: \C^d \to \C$ be a non-zero homogeneous polynomial of degree $k,$ and let $I_m^a(P_k)$, $S_m^a(P_k)$  denote respectively the infimum and the supremum appearing in (\ref{LowerAsymB}) and (\ref{upperAsymB}) when the apolar norm is used. Let $\rho < k$ be 
	a natural number.
	If   the polynomials contained in
	$$
	\{\partial^{\alpha}P_k: \left\vert \alpha\right\vert =\rho\}
	$$ 
	do not have a
	common complex zero $c\in \C^d \setminus \{0\},$ then there exists a $C_1>0$ such that for all $m\geq
	1$, 
	\begin{equation}\label{inf}
		I_m^a(P_k)
		\ge 
		C_1 \sqrt{m^{k - \rho}}.
	\end{equation}
	On the other hand, if there is 
	a  
	$c\in\mathbb{C}^{d} \setminus \{0\}$ with
	\[
	\left(  \partial^{\alpha}P_k\right)  \left(  c\right)  =0\text{ for all
		multi-indices $\alpha$ such that }\left\vert \alpha\right\vert =\rho,
	\] 
	then there exists a $C_2 > 0$  such that for all $m\geq
	1$ we have
	\begin{equation}\label{sup}
		S_m^a(P_k) 
		\le 
		C_2 \sqrt{m^{k - \rho - 1}}.
	\end{equation}
\end{corollary}

Note that by \cite[Lemma 13]{AlRe23} (a result essentially due to B. Beauzamy) we always have 
\begin{equation}\label{supBB}
	S_m^a(P_k)
	\le 
	C \sqrt{m^{k}}.
\end{equation}
 
We are indebted to Prof. A. Bravo for useful information regarding Gr\"obner bases.

\section{Background definitions and results}

 A polynomial $P\left(  z\right)  $
is \emph{homogeneous} of degree $m$ if $P\left(  tz \right)  =t^{m}P\left(
z\right)  $ for all $t>0$ and for all $z$. In particular, the constant zero polynomial is homogeneous of every degree.  We use the standard notations for multi-indices
$\alpha=\left(  \alpha_{1}, \dots ,\alpha_{d}\right)  \in\mathbb{N}^{d}$,
namely $\alpha!=\alpha_{1}! \cdots \alpha_{d}!$, $\left\vert \alpha\right\vert
=\alpha_{1}+\cdots+\alpha_{d}$, and $z^{\alpha}= z_{1}^{\alpha_{1}} \cdots z_{d}^{\alpha_{d}}.$ Recall that  $0 \in \N$; some authors use
$\N_0$ to emphasize this fact.

A homogeneous polynomial $f_{m}$ of degree $m$
can be written as
\[
f_{m} (z) 
=
{\displaystyle\sum_{\left\vert \alpha\right\vert =m}}
c_{\alpha} z^{\alpha}.
\]
The Bombieri norm, introduced in \cite{BBEM90}, is defined for $f_{m}
\in\mathcal{P}_{m}\left(  \mathbb{C}^{d}\right)  $ by
\[
\left[  f_{m}\right]  _{B}=\frac{1}{\sqrt{m!}}\sqrt{
{\displaystyle\sum_{\left\vert \alpha\right\vert =m}}
\left\vert c_{\alpha}\right\vert ^{2}\alpha!},
\]
Bombieri's inequality for $f_{m}\in\mathcal{P}_{m}\left(  \mathbb{C}
^{d}\right)  $ and $P_{k}\in\mathcal{P}_{k}\left(  \mathbb{C}^{d}\right)  $
says that
\begin{equation}
\left[  P_{k}\cdot f_{m}\right]  _{B}\geq\frac{\sqrt{k!m!}}{\sqrt{\left(
k+m\right)  !}}\left[  P_{k}\right]  _{B}\left[  f_{m}\right]  _{B}.
\label{ineqBombieri}
\end{equation} Our main result is formulated in terms of the closely related 
 \emph{apolar norm}, induced by the following inner
product for polynomals: if
\[
P\left(  z\right)  =\sum_{\alpha\in\mathbb{N}^{d},\left\vert
\alpha\right\vert \leq N}c_{\alpha}z^{\alpha}\text{ and }Q\left(  z\right)
=\sum_{\alpha\in\mathbb{N}^{d},\left\vert \alpha\right\vert \leq
M}d_{\alpha}z^{\alpha}
\]
then the \emph{apolar inner product} $\left\langle \cdot,\cdot\right\rangle_{a}$ on $\mathcal{P}\left(  \mathbb{C}^{d}\right)  $
is defined by
\begin{equation}
\left\langle P,Q\right\rangle _{a}
:=
\sum_{\alpha\in\mathbb{N}^{d}}
\alpha!c_{\alpha}\overline{d_{\alpha}}, \label{eqQPD}
\end{equation}
and $\left\Vert f\right\Vert _{a}:=\sqrt{\left\langle
f,f\right\rangle _{a}}.$ The apolar inner product appeared within the Theory of Invariants in the XIX-th century. Nowadays it is likely to be found, among other places, in textbooks dealing with spherical harmonics.

For $f_{m}\in\mathcal{P}_{m}\left(  \mathbb{C}
^{d}\right)  $ the simple relationship
\[
\text{ }\left[  f_{m}\right]  _{B}^{2}=\frac{1}{m!}\left\Vert f_{m}\right\Vert
_{a}^{2}
\]
holds. As mentioned above, Bombieri's inequality for the
apolar norm is
\[
\left\Vert P_{k}f_{m}\right\Vert _{a}\geq\left\Vert P_{k}\right\Vert
_{a}\left\Vert f_{m}\right\Vert _{a}.
\]

In later sections we shall use the {\em Newman-Shapiro identity}
\begin{equation}
\left\Vert PQ\right\Vert _{\alpha}^{2}=\sum_{\gamma}\frac{1}{\gamma
!}\left\Vert \left(  \partial^{\gamma}P^{\ast}\right)  \left(  D\right)
Q\right\Vert _{a}^{2},\label{eqBD2}
\end{equation}
which was announced in \cite{NeSh66}, and rediscovered by B. Reznick, cf. \cite{Rezn93}; for a
short proof see \cite{Zeil}.  Actually, the Newman-Shapiro identity (\ref{eqBD2}) is stated in \cite[Theorem 3]{NeSh66} not just for polynomials, but for $Q$ entire with no other restriction, and for $P$ entire satisfying a certain growth condition, namely, for every $\varepsilon > 0$,  $P(z) = O (e^{\varepsilon |z|^2})$.

\vskip .2 cm

D.J. Newman and H.S. Shapiro proved also a Fischer type decomposition (see
\cite{NeSh66} for the precise statement) for the \emph{Bargmann-Segal space}
$\mathcal{F}_{d}$, also known as the \emph{Fock space}, especially among physicists,  and as
the \emph{Fischer space}, cf. \cite{Barg61}.  This space, the completion of the pre-Hilbert space of complex polynomials in $d$ variables under the apolar norm, is defined as the set of all entire functions
$f:\mathbb{C}^{n}\rightarrow\mathbb{C}$ that satisfy
\begin{equation}
\left\Vert f\right\Vert _{F_{d}}^{2}:=\frac{1}{\pi^{d}}\int_{\mathbb{C}^{d}
}\left\vert f\left(  z\right)  \right\vert ^{2}e^{-\left\vert z\right\vert
^{2}}dz<\infty,
\end{equation}
where $dz$ is Lebesgue measure on $\mathbb{R}^{2d}$. In \cite{Barg61} the
identity 
\begin{equation}
\left\langle P,Q\right\rangle _{a}=\frac{1}{\pi^{d}}\int_{\mathbb{R}^{d}}
\int_{\mathbb{R}^{d}}P\left(  x+iy\right)  \overline{Q\left(  x+iy\right)
}e^{-\left\vert x\right\vert ^{2}-\left\vert y\right\vert ^{2}}dxdy<\infty
\label{Bargman}
\end{equation}
for all $P,Q\in\mathcal{P}\left(  \mathbb{R}^{d}\right)  $ is mentioned where
$dxdy$ is Lebesgue measure on $\mathbb{R}^{2d}$ (this result was rediscovered
in \cite{Pina12}). 
 
\section{Proof of the sufficiency part of Theorem \ref{ThmMain1}}

The next result is known, cf.  \cite[Lemma 10]{KhSh92}. We present a simpler,   more elementary proof.

\begin{lemma} \label{derivative}
Let $P$ be a polynomial  of degree $\leq k$, and denote by $M_j$ the degree of $z_j$ in $P$. Then for all polynomials $Q$ we have 
\begin{equation}\label{lemma10KS}
\left\Vert \left( \partial_{j} P \right) \cdot Q\right\Vert_{a}
\le
M_j \left\Vert P\cdot Q\right\Vert _{a}
\le
k \left\Vert P\cdot Q\right\Vert _{a}.
\end{equation}
\end{lemma}

\begin{proof} Given a multiindex $\alpha =\left( \alpha _{1},\dots,\alpha
_{d}\right) $,  define 
$$
\alpha \left( j \right) : = \left( \alpha _{1},\dots, \alpha _{j} +1 ,\dots, \alpha
_{d}\right), 
$$
and note that  $\partial^{\alpha} \partial_{j}
 = \partial^{\alpha(j)}$. Furthermore, if $\alpha_j + 1 > M_k$, then $\partial^{\alpha(j)} P = 0$.
 
 Thus, applying  the Newman-Shapiro identity (\ref{eqBD2}) to 
 $
 \partial_{j} P  \cdot Q
 $
 and to
 $
 P \cdot Q
 $
 we obtain
 \begin{equation*}
\left\Vert \partial_{j} P  \cdot Q\right\Vert _{a}^{2}
=
\sum_{\alpha }\frac{1}{\alpha !}\left\Vert  \partial^{\alpha} \partial_{j} P \left( D\right) Q\right\Vert _{a}^{2}
\end{equation*}
and
\begin{equation*}
\left\Vert P  \cdot Q\right\Vert _{a}^{2}
=
\sum_{\beta }\frac{1}{\beta !}\left\Vert \partial^{\beta}  P \left( D\right) Q\right\Vert _{a}^{2}
\end{equation*}
respectively.
 Hence
 \begin{equation*}
\left\Vert \partial_{j} P  \cdot Q\right\Vert _{a}^{2}
=
\sum_{\alpha }\frac{\alpha_j + 1}{\alpha (j) !}\left\Vert  \partial^{\alpha} \partial_{j} P \left( D\right) Q\right\Vert _{a}^{2}
\le
M_j \sum_{\alpha }\frac{1}{\alpha (j) !}\left\Vert  \partial^{\alpha(j)} P \left( D\right) Q\right\Vert _{a}^{2}
\end{equation*}
\begin{equation*}
\le M_j  \sum_{\beta}\frac{1}{\beta !}\left\Vert  \partial^{\beta} P \left( D\right) Q\right\Vert _{a}^{2}
= 
M_j \left\Vert P  \cdot Q\right\Vert _{a}^{2}.
\end{equation*}
\end{proof}

We recall the sufficiency part of Theorem \ref{ThmMain1}.

\begin{theorem} \label{mainbis}
Let $P$ be a homogeneous polynomial of degree $k \ge 1.$ If there is a
natural number $\rho < k$ such that the  polynomials in $\{\partial^{\alpha}P:$
$\left\vert \alpha\right\vert =\rho\}$ do not have a common complex zero
$c\neq0,$ then the Khavinson-Shapiro exponent $l\left(  P\right)  $ is at
least $k-\rho.$
\end{theorem}

\begin{proof}
Since the functions $\partial^{\alpha}P$ with $\left\vert \alpha\right\vert
=\rho$ do not have a common nontrivial complex zero, it follows that for all
$\eta\in\mathbb{S}^{2d-1}=\left\{  z\in\mathbb{C}^{d}:\left\vert z\right\vert
=1\right\}$, the following expression is strictly positive:
\[
{\displaystyle\sum_{\left\vert \alpha\right\vert =\rho}}
\left\vert \partial^{\alpha}P\left(  \eta\right)  \right\vert ^{2}>0.
\]
Hence there exists a minimum on $\mathbb{S}^{2d-1}$ which is larger than some
constant $C = C(P, \rho) >0.$ Let $q_{m}\left(  x\right)  $ be a homogeneous polynomial of
degree $m.$ Then
\[
{\displaystyle\sum_{\left\vert \alpha\right\vert =\rho}}
\left\vert \partial^{\alpha}P\left(  \eta\right)  \right\vert ^{2}\left\vert
q_{m}\left(  \eta\right)  \right\vert ^{2}\geq C\left\vert q_{m}\left(
\eta\right)  \right\vert ^{2}.
\]
Integrate the inequality with respect to the area measure $d\eta$  on the unit
sphere $\mathbb{S}^{2d-1}$ to get
\[
{\displaystyle\sum_{\left\vert \alpha\right\vert =\rho}}
\ \int_{\mathbb{S}^{2d-1}}\left\vert \partial^{\alpha}P\left(  \eta\right)
\right\vert ^{2}\left\vert q_{m}\left(  \eta\right)  \right\vert ^{2}d\eta\geq
C\int_{\mathbb{S}^{2d-1}}\left\vert q_{m}\left(  \eta\right)  \right\vert
^{2}d\eta.
\]
Next we use identity (\ref{Bargman}), integration in polar coordinates, homogeneity,  the change of variable $t = r^2$, and some basic properties of the Gamma function, to obtain
\[ 2 \pi^d {\displaystyle\sum_{\left\vert \alpha\right\vert =\rho}}
\left\Vert \partial
^{\alpha}P\cdot q_{m}\right\Vert _{a}^{2}
$$
$$
=
2 {\displaystyle\sum_{\left\vert \alpha\right\vert =\rho}}
\int_{0}^{\infty
}r^{2m+2\left(  k-\rho\right)  }e^{-r^{2}}r^{2d-1}dr \ \int_{\mathbb{S}^{2d-1}}\left\vert \partial^{\alpha}P\left(  \eta\right)
\right\vert ^{2}\left\vert q_{m}\left(  \eta\right)  \right\vert ^{2}d\eta
$$
$$
=
\Gamma\left(m+ d +  k-\rho  \right)\ {\displaystyle\sum_{\left\vert \alpha\right\vert =\rho}} \int_{\mathbb{S}^{2d-1}}\left\vert \partial^{\alpha}P\left(  \eta\right)
\right\vert ^{2}\left\vert q_{m}\left(  \eta\right)  \right\vert ^{2}d\eta
$$
$$
\geq
\frac{\Gamma\left(m+ d +  k-\rho  \right)}{\Gamma\left(m + d \right)} \ 
\Gamma\left(m + d \right) C\int_{\mathbb{S}^{2d-1}}\left\vert q_{m}\left(  \eta\right)  \right\vert
^{2}d\eta
$$
$$
\ge
2 \pi^d (m + d)^{k - \rho} C  \left\Vert q_{m}\right\Vert
_{a}^{2}.
\]

By Lemma \ref{derivative}, 
 given a polynomial $P$ of degree $\leq k$, for all
polynomials $Q$ the estimate
\begin{equation}\label{lemma10KS}
\left\Vert \left( \partial_{j} P \right) \cdot Q\right\Vert_{a}
\le
k \left\Vert P\cdot Q\right\Vert _{a}
\end{equation}
holds, which by repeated application  implies that
\[
\left\Vert \partial^{\alpha}P\cdot Q\right\Vert _{a}
\leq
k^{\left\vert \alpha\right\vert }\left\Vert P\cdot Q\right\Vert
_{a}.
\]
Using the fact that the cardinality of the set of all multi-indices $\alpha$ with $\left\vert \alpha\right\vert =\rho$ is
equal to $\binom{d+\rho-1}{\rho}$, we see that
\begin{align*}
\binom{d+\rho-1}{\rho}\left\Vert P\cdot q_{m}\right\Vert _{a}^{2}  &  \geq
\frac{1}{k^{2\left\vert \alpha\right\vert }}
{\displaystyle\sum_{\left\vert \alpha\right\vert =\rho}}
\left\Vert \partial^{\alpha}P\cdot q_{m}\right\Vert _{a}^{2}
\geq
\frac{C\left(  m+d\right)  ^{k-\rho}}{k^{2\left\vert \alpha\right\vert }
}\left\Vert q_{m}\right\Vert _{a}^{2}.
\end{align*}
\end{proof}

It follows from the preceding proof that the best constant appearing in the Khavinson-Shapiro
bounds (when norms are squared) is at least as large as
$$
\min_{\eta\in\mathbb{S}^{2d-1}}
k^{ - 2\rho} \binom{d+\rho-1}{\rho}^{-1}
	{\displaystyle\sum_{\left\vert \alpha\right\vert =\rho}}
	\left\vert \partial^{\alpha}P\left(  \eta\right)  \right\vert ^{2}.
$$

\section{Proof of the necessity part of Theorem \ref{ThmMain1} }
A fundamental property of the apolar inner product lies in the fact 
  that the adjoint
 of the multiplication operator $M_{Q}\left(  g\right)  =Qg$, is the
differential operator associated to the polynomial $Q^{\ast},$ i.e.,
\begin{equation}
\left\langle Q^{\ast}\left(  D\right)  f,g\right\rangle _{a}=\left\langle
f,Q\cdot g\right\rangle _{a}\label{eqFischeradjoint}.
\end{equation}
The next two results are known. We include  proofs for the reader's convenience (for the first see e.g.  \cite[Formula (5.8)]{Rezn93}). For the second we are not able to provide a reference.

\begin{proposition}
\label{Prop1}For each natural number $m$ and $c, z \in\mathbb{C}^{d}$ we have
\[
\left\Vert \left\langle  z, c \right\rangle_2 ^{m}\right\Vert _{a}=\sqrt
{m!}\left\vert c\right\vert ^{m}.
\]
\end{proposition}

\begin{proof}
Let $c= (c_1, \dots , c_d)$. When $m=1$ the statement means that $\left\Vert \left\langle z, c \right\rangle_2
\right\Vert _{a}^{2}=
{\displaystyle\sum_{j=1}^{d}}
\left\vert c_{j}\right\vert^{2}$, which is obviously true. Assume the
result holds for $m.$ Define $L_{{c}}
=
{\displaystyle\sum_{j=1}^{d}}
{c_{j}}\frac{\partial}{\partial z_{j}}.$ It is easy to see that
$L_{{c}}\left( \left\langle z, c \right\rangle_2^{m+1}\right)
=
\left(  m+1\right)
\left\vert c\right\vert^{2}\left\langle z ,c \right\rangle_2^{m}$. Using
(\ref{eqFischeradjoint}) we obtain
\begin{align*}
\left\Vert \left\langle z ,c \right\rangle_2^{m+1}\right\Vert _{a}^{2}  &
=\left\langle \left\langle z ,c \right\rangle_2 ^{m+1},\left\langle
z , c \right\rangle_2^{m+1}\right\rangle _{a}=\left\langle L_{{c}}\left\langle z, c \right\rangle_2 ^{m+1},\left\langle z,c\right\rangle_2
^{m}\right\rangle_{a}\\
&  =\left(  m+1\right)  \left\vert c\right\vert ^{2}\left\Vert \left\langle
z, c \right\rangle_2 ^{m}\right\Vert _{a}^{2}=\left(  m+1\right)  !\left\vert
c\right\vert ^{2\left(  m+1\right)  }.
\end{align*}

\end{proof}

\begin{proposition}
\label{Prop2} Let $P\left(  z\right)  $ be a non-zero homogeneous polynomial of degree
$k$ and let $c\in\mathbb{C}^{d}\setminus \{0\}$. Then for every $m\geq k$, 
\[
P\left(  D\right)  \left(  \left\langle z, c \right\rangle_2^{m}\right)
=\frac{m!}{\left(  m-k\right)  !}\left\langle z, c \right\rangle_2 ^{m-k}P\left(\overline{c}\right)
\]
and
\[
\left\Vert P\left(  D\right)  \left(  \left\langle z, c \right\rangle_2
^{m}\right)  \right\Vert _{a}=\left\vert c\right\vert ^{-k}\sqrt{\frac
{m!}{\left(  m-k\right)  !}}\left\Vert \left\langle z ,c \right\rangle_2
^{m}\right\Vert _{a}\left\vert P\left(  \overline{c} \right)  \right\vert .
\]

\end{proposition}

\begin{proof}
Note that $\frac{\partial}{\partial z_{j}}\left\langle z,c\right\rangle_2
^{m}=m\left\langle z , c \right\rangle_2 ^{m-1} \overline{c_{j}},$ so for $\alpha = ({\alpha_{1}}, \dots ,{\alpha_{d}})$ with 
$|\alpha| = k$, iteration gives
\[
D^{\alpha}\left\langle  z ,c \right\rangle_2^{m}
=
m\left(  m-1\right)
\cdots\left(  m-k+1\right)  \overline{c_{1}}^{\alpha_{1}}\cdots \overline{c_{d}}^{\alpha_{d}}\left\langle z, c\right\rangle_2^{m-k}
$$
$$
=
\frac{m!}{\left(  m-k\right)
!} \ \overline{c} ^{\alpha}\left\langle  z , c \right\rangle_2 ^{m-k}.
\]
Let 
$P\left(  z\right)  
=
{\displaystyle\sum_{\left\vert \alpha\right\vert =k}}
a_{\alpha}z^{\alpha}.$ By linearity
\[
P\left(  D\right)  \left(  \left\langle z ,c \right\rangle_2 ^{m}\right)
=\frac{m!}{\left(  m-k\right)  !}\left\langle z ,c \right\rangle_2 ^{m-k}
{\displaystyle\sum_{\left\vert \alpha\right\vert =k}}
a_{\alpha} \overline{c}^{\alpha}
=
\frac{m!}{\left(  m-k\right)  !}\left\langle
z ,c \right\rangle_2 ^{m-k}P\left(  \overline{c}\right).
\]
Thus the first statement is proven. Taking norms squared on both sides yields
\[
\left\Vert P\left(  D\right)  \left(  \left\langle  z , c\right\rangle
^{m}\right)  \right\Vert _{a}^{2}=\left\vert P\left(  \overline{c} \right)  \right\vert
^{2}\left(  \frac{m!}{\left(  m-k\right)  !}\right)  ^{2}\left\Vert
\left\langle z ,c \right\rangle_2 ^{m-k}\right\Vert _{a}^{2}.
\]
An application of Proposition \ref{Prop1} finishes the proof, since
\[
\left(  \frac{m!}{\left(  m-k\right)  !}\right)  ^{2}\left\Vert \left\langle
z,c\right\rangle_2^{m-k}\right\Vert _{a}^{2}=\frac{m!}{\left(  m-k\right)
!}\left\Vert \left\langle z, c \right\rangle_2 ^{m}\right\Vert _{a}^{2}\left\vert
c\right\vert ^{-2k}.
\]

\end{proof}

\begin{theorem}
\label{Thm3}Let $P\left(  z\right)  $ be a non-zero homogeneous polynomial of degree
$k$ and let $c\in\mathbb{C}^{d}\setminus \{0\}$. Then for every natural number $m\geq k$,
\[
\left\Vert P\left(  z \right)  \left\langle z,c \right\rangle_2 ^{m}\right\Vert
_{a}^{2}
=
\left\Vert \left\langle z ,c \right\rangle_2 ^{m}\right\Vert _{a}^{2}
{\displaystyle\sum_{\left\vert \alpha\right\vert \leq k}}
\frac{m!\left\vert \left(  \partial^{\alpha}P^{\ast}\right)  \left(  \overline{c} \right)
\right\vert ^{2}}{\alpha!\left(  m-\left(  k-\left\vert \alpha\right\vert
\right)  \right)  !}\frac{\left\vert c\right\vert ^{2\left\vert \alpha
\right\vert }}{\left\vert c\right\vert ^{2k}}.
\]
\end{theorem}

\begin{proof}
The Newman-Shapiro identity (\ref{eqBD2}) shows that
\[
\left\Vert P\left(  z \right)  \left\langle z, c \right\rangle_2 ^{m}\right\Vert
_{a}^{2}
=
{\displaystyle\sum_{\left\vert \alpha\right\vert \leq k}}
\frac{1}{\alpha!}\left\Vert \left(  \partial^{\alpha}P^{\ast}\right)  \left(
D\right)  \left\langle z ,c \right\rangle_2^{m}\right\Vert _{a}^{2}.
\]
Note that $\partial^{\alpha}P^{\ast}$ has degree $k-\left\vert \alpha
\right\vert .$ Proposition \ref{Prop2} applied to $\partial^{\alpha}P^{\ast}$
gives
\[
\left\Vert \partial^{\alpha}P^{\ast}\left(  D\right)  \left(  \left\langle
z,c\right\rangle_2 ^{m}\right)  \right\Vert _{a}^{2}
=
\left\vert \partial
^{\alpha}P^{\ast}\left(  \overline{c} \right)  \right\vert ^{2}\left\vert c\right\vert
^{-2\left(  k-\left\vert \alpha\right\vert \right)  }\frac{m!}{\left(
m-\left(  k-\left\vert \alpha\right\vert \right)  \right)  !}\left\Vert
\left\langle z, c\right\rangle_2 ^{m}\right\Vert _{a}^{2}.
\]
These two equations imply the result.
\end{proof}

\begin{lemma}
\label{Lem1} Let $P\left(  z \right)  $ be a homogeneous polynomial of degree
$k \ge 1$. Fix $c\in\mathbb{C}^{d}\setminus \{0\}$. If there exists an $r\in\mathbb{N}$ with $r \le k$ and
\begin{equation}
\left(  \partial^{\alpha}P\right)  \left(  c \right)  =0\text{ for all
}\left\vert \alpha\right\vert =r, \label{eqzero}
\end{equation}
then $\left(  \partial^{\alpha}P\right)  \left(  c\right)  =0$ for all
$\left\vert \alpha\right\vert \leq r.$
\end{lemma}

\begin{proof}
Let us recall that for a homogeneous polynomial $P$ of degree $k$, Euler's
lemma states that
\[
k P\left(   z \right)  
=
{\displaystyle\sum_{j=1}^{d}}
z_{j}\frac{\partial}{\partial z_{j}}P\left(  z \right).
\]
Assume now that (\ref{eqzero}) holds, and let $\beta\in\mathbb{N}^{d}$
be such that $\left\vert \beta\right\vert =r-1.$ Then $\partial^{\beta}P\left(
z \right)  $ is a homogeneous polynomial of degree $k-\left\vert \beta
\right\vert $, so by  Euler's lemma, 
\[
\left(  k-\left\vert \beta\right\vert \right)  \partial^{\beta}P\left(
 z \right)  
=
{\displaystyle\sum_{j=1}^{d}}
z_{j} \frac{\partial}{\partial z_{j}}\partial^{\beta}P\left(  z \right).
\]
By assumption, 
$$\frac{\partial}{\partial z_{j}}\partial^{\beta}P\left(  c \right)  =0
$$ for  $j=1,\dots , d.$ It follows
that $\partial^{\beta}P\left(  c\right)  =0$ for all $\beta\in\mathbb{N}^{d}$ with $\left\vert \beta\right\vert =r-1.$ Proceeding inductively we
see that $\left(  \partial^{\alpha}P\right)  \left(  c\right)  =0$ for all
$\left\vert \alpha\right\vert \leq r.$
\end{proof}

\begin{corollary}
Let $P\left(  z \right)  $ be a homogeneous polynomial of degree $k \ge 1.$ If there
exists a constant $c\in\mathbb{C}^{d}\setminus \{0\}$ with
\[
\left(  \partial^{\alpha}P^{\ast}\right)  \left(  \overline{c}  \right)  =0\text{ for all $\alpha$ such that 
}\left\vert \alpha\right\vert =k-1,
\]
then
\[
\left\Vert P\left(  z \right)  \left\langle z , c \right\rangle_2 ^{m}\right\Vert
_{a}^{2}=\left\Vert P\right\Vert _{a}^{2}\left\Vert \left\langle
z ,c \right\rangle_2 ^{m}\right\Vert _{a}^{2}\ .
\]

\end{corollary}

\begin{proof}
Theorem \ref{Thm3} together with Lemma \ref{Lem1} and the assumption of the Corollary show that
\[
\left\Vert P\left(  z \right)  \left\langle z ,c \right\rangle_2 ^{m}\right\Vert
_{a}^{2}=\left\Vert \left\langle z , c\right\rangle_2 ^{m}\right\Vert _{a}^{2}
{\displaystyle\sum_{\left\vert \alpha\right\vert =k}}
\frac{1}{\alpha!}\left\vert \left(  \partial^{\alpha}P^{\ast}\right)  \left(
\overline{c}  \right)  \right\vert ^{2}.\
\]
For
$\left\vert \alpha\right\vert =k$,  the expression $\left(  \partial^{\alpha}P^{\ast}\right)  \left(  z\right)  $  is polynomial of degree $0,$ so it is
constant. Hence 
\[
{\displaystyle\sum_{\left\vert \alpha\right\vert =k}}
\frac{1}{\alpha!}\left\vert \left(  \partial^{\alpha}P^{\ast}\right)  \left(
 \overline{c}  \right)  \right\vert ^{2}=\
{\displaystyle\sum_{\left\vert \alpha\right\vert =k}}
\frac{1}{\alpha!}\left\vert a_{\alpha}\alpha!\right\vert ^{2}
=
{\displaystyle\sum_{\left\vert \alpha\right\vert =k}}
\alpha!\left\vert a_{\alpha}\right\vert ^{2}=\left\Vert P\right\Vert _{a}
^{2}.
\]

\end{proof}

Note that for $k \ge 1$ and
$\left\vert \alpha\right\vert =k-1$, the polynomial $\left(  \partial^{\alpha}P^{\ast}\right)  \left(  z \right)  $  has  degree $1$, so it can
be written as
\[
\left(  \partial^{\alpha}P^{\ast}\right)  \left(  z \right)  =\left\langle z, 
\sigma_{\alpha}\right\rangle_2
\]
where $\sigma_{\alpha}\in\mathbb{C}^{d}\setminus \{0\}.$ Consider the linear span
$V_{k-1}\left(  P\right)  $ of the vectors $\sigma_{\alpha}$ with $\left\vert
\alpha\right\vert =k-1.$ If this linear space has dimension $\leq d-1,$ we can
find a nonzero vector $\overline{c}  \in\mathbb{C}^{d}$  orthogonal to $V_{k-1}\left(
P\right),$ thus orthogonal to all the vectors $\sigma_{\alpha}$ with $\left\vert
\alpha\right\vert =k-1.$ It follows that $\partial^{\alpha}P^{\ast}\left(
 \overline{c}  \right)  =0$ for all $\alpha$ with $\left\vert \alpha\right\vert =k-1.$
Conversely, the condition $\partial^{\alpha}P^{\ast}\left(  \overline{c} \right)  =0$ for all $\alpha$
with $\left\vert \alpha\right\vert =k-1$ and a nonzero $\overline{c} $, implies that the linear span
$V_{k-1}\left(  P\right)  $ has dimension $\leq d-1.$ Thus we have
 
\begin{corollary} \label{ell0}
Let $P\left(  z \right)  $ be a homogeneous polynomial of degree $k \ge 1$ such that
the system of linear functions $\partial^{\alpha}P^{\ast}$ with $\left\vert
\alpha\right\vert =k-1$ has rank $\leq d-1.$ Then there exists a ${c}\in
\mathbb{C}^{d}\setminus \{0\}$ such that
\[
\left\Vert P\left(  z \right)  \left\langle z, c \right\rangle_2^{m}\right\Vert
_{a}^{2}=\left\Vert P\right\Vert _{a}^{2}\left\Vert \left\langle
z, c \right\rangle_2 ^{m}\right\Vert _{a}^{2},\
\]
and thus, the Khavinson-Shapiro exponent $l\left(P\right)  $ is zero.
\end{corollary}

Finally, we prove the necessity part of Theorem \ref{ThmMain1}, restated next: 

\begin{theorem}
Let $P\left(  z \right)  $ be a homogeneous polynomial of degree $k \ge 2$ and let $\rho$ be
a natural number with $\rho\leq k-1.$ If there exists a ${c}\in\mathbb{C}^{d}\setminus \{0\}$
with
\[
\left(  \partial^{\alpha}P^{\ast}\right)  \left(  \overline{c}\right)  =0\text{ for all
}\left\vert \alpha\right\vert =\rho,
\]
then there exists a constant $C > 0$  depending only on $P$ and $c$, such that for
all $m\geq k$,
\[
\left\Vert P\left(  z \right)  \left\langle z ,c \right\rangle ^{m}\right\Vert
_{a}^{2}\leq C\left\Vert \left\langle z,c\right\rangle ^{m}\right\Vert
_{a}^{2}\ \cdot m^{k-1-\rho}.
\]

\end{theorem}

\begin{proof}
According to Theorem \ref{Thm3}, for every natural number $m\geq k$
\[
\left\Vert P\left(  z\right)  \left\langle z,c\right\rangle_2 ^{m}\right\Vert
_{a}^{2}=\left\Vert \left\langle  z , c \right\rangle_2 ^{m}\right\Vert _{a}^{2}
{\displaystyle\sum_{\left\vert \alpha\right\vert \leq k}}
\frac{m!\left\vert \left(  \partial^{\alpha}P^{\ast}\right)  \left(  \overline{c}\right)
\right\vert ^{2}}{\alpha!\left(  m-\left(  k-\left\vert \alpha\right\vert
\right)  \right)  !}\frac{\left\vert c\right\vert ^{2\left\vert \alpha
\right\vert }}{\left\vert c\right\vert ^{2k}}.
\]
Since $\left(  \partial^{\alpha}P^{\ast}\right)  \left(  \overline{c} \right)  =0$ for all
$\left\vert \alpha\right\vert =\rho$ it follows by Lemma \ref{Lem1} that
$\left(  \partial^{\alpha}P^{\ast}\right)  \left(  \overline{c}\right)  =0$ for all
$\left\vert \alpha\right\vert \leq\rho.$ Hence
\begin{align*}
\left\Vert P\left(  z \right)  \left\langle z , c \right\rangle_2^{m}\right\Vert
_{a}^{2}  &  \leq\left\Vert \left\langle z ,c \right\rangle ^{m}\right\Vert
_{a}^{2}
{\displaystyle\sum_{\rho<\left\vert \alpha\right\vert \leq k}}
\frac{m!\left\vert \left(  \partial^{\alpha}P^{\ast}\right)  \left( \overline{c} \right)
\right\vert ^{2}}{\alpha!\left(  m-\left(  k-\left\vert \alpha\right\vert
\right)  \right)  !}\frac{\left\vert c\right\vert ^{2\left\vert \alpha
\right\vert }}{\left\vert c\right\vert ^{2k}}\\
&  \leq\left\Vert \left\langle z , c\right\rangle_2^{m}\right\Vert _{a}^{2}\cdot
m^{k-\rho-1}{\displaystyle\sum_{\rho<\left\vert \alpha\right\vert \leq k}}
\frac{\left\vert \left(  \partial^{\alpha}P^{\ast}\right)  \left(   \overline{c}\right)
\right\vert ^{2}}{\alpha!}\frac{\left\vert c\right\vert ^{2\left\vert
\alpha\right\vert }}{\left\vert c\right\vert ^{2k}},
\end{align*}
and we take the constant $C$ to be the last sum.
\end{proof}

Recall that the expression for $C$ can be simplified by taking $|c| = 1$, which is always possible by homogeneity.

\vskip .2 cm

So far we have worked with polynomials having complex coefficients. If $x \in \R^d$ and the
polynomial $P\left(  x\right)$ has real coefficients, it is natural to ask whether there exist polynomials
$q_{m}\left(  x\right)$  with {\em real  coefficients} for which the conclusion of the preceeding result holds. The answer is positive, as can be seen  simply by considering the real
and imaginary parts of the complex polynomial  $\left\langle z ,c\right\rangle_2^m$, $c\in\mathbb{C}^{d}\setminus \{0\}$.  

\begin{corollary}
Let $P\left(  x\right)  $ be a homogeneous polynomial of degree $k$ with real
coefficients, and let $\rho\leq k-1$ be a natural number.  If there exists a constant
$c\in\mathbb{C}^{d}\setminus \{0\}$ with
\[
\left(  \partial^{\alpha}P\right)  \left(  c\right)  =0\text{ for all
}\left\vert \alpha\right\vert =\rho,
\]
then for every $m \ge k$ there exist a homogeneous polynomial $q_{m}$ of degree $m$ with real
coefficients, and a constant  $C>0$, such that
\[
\left\Vert P  \cdot q_{m}\right\Vert _{a}^{2}\leq C\left\Vert
q_{m}\right\Vert _{a}^{2}\cdot m^{k-\rho-1}.
\]
\end{corollary}

\begin{proof} Suppose we have a homogeneous polynomial $f(x)$ with complex coefficients, degree $m$, and $x\in \R^d$, say  $f (x) =
{\displaystyle\sum_{\left\vert \alpha\right\vert =m}}
c_{\alpha}x^{\alpha}$, where $c_{\alpha}=a_{\alpha}+ib_{\alpha}$. Then we can write $f = f_1 + i f_2$, where $f_1$ and $f_2$ are real polynomials. Furthermore, 
\begin{align*}
\left\Vert f\right\Vert _{a}^{2}  &  
=
{\displaystyle\sum_{\left\vert \alpha\right\vert =m}}
\left\vert c_{\alpha}\right\vert ^{2}\alpha!=
{\displaystyle\sum_{\left\vert \alpha\right\vert =m}}
a_{\alpha}^{2}\alpha!+
{\displaystyle\sum_{\left\vert \alpha\right\vert =m}}
b_{\alpha}^{2}\alpha!\\
&  =\left\Vert \operatorname{Re}f\right\Vert _{a}^{2}+\left\Vert
\operatorname{Im}f\right\Vert _{a}^{2}.
\end{align*}
Hence  we must have either $\left\Vert \operatorname{Re}f\right\Vert
_{a}^{2}\geq\frac{1}{2}\left\Vert f\right\Vert _{a}^{2}$ or $\left\Vert
\operatorname{Im}f\right\Vert _{a}^{2}\geq\frac{1}{2}\left\Vert f\right\Vert
_{a}^{2}$. We apply this observation to $f(x) = \left\langle x,c\right\rangle_2 ^{m}$. Now
\[
\operatorname{Re}\left(  \left\langle x,c\right\rangle_2 ^{m}\right)  =\frac
{1}{2}\left(  \left\langle x ,c\right\rangle_2 ^{m}+\left\langle c, x\right\rangle_2^{m}\right)  \text{ and }\operatorname{Im}\left(
\left\langle x,c\right\rangle_2 ^{m}\right)  
=
\frac{1}{2i}\left(  \left\langle x,c\right\rangle_2^{m} - \left\langle c, x \right\rangle_2^{m}\right)  .
\]

Note that $\left\Vert f\right\Vert _{a}=\left\Vert f^{\ast
}\right\Vert _{a}$ for every polynomial $f$. Since $P=P^{\ast}$, by the triangle inequality and the preceding observation we have
\[
\left\Vert P\left(  x \right)  \cdot\operatorname{Re}
\left( \left\langle x,c\right\rangle_2^{m}\right)  \right\Vert _{a}
\leq
\frac{1}{2}\left\Vert
P\left(  x\right)  \left\langle x,c\right\rangle_2 ^{m}\right\Vert _{a}
+
\frac{1}{2}\left\Vert P\left(  x\right)  \left\langle 
{c},x\right\rangle_2 ^{m}\right\Vert _{a}
\]
\[ =
\left\Vert P\left(
x\right)  \left\langle x,c\right\rangle_2 ^{m}\right\Vert _{a}
\leq
C^{1/2}\left\Vert \left\langle x ,c\right\rangle_2 ^{m}\right\Vert _{a}\cdot
m^{(k-\rho-1)/2}.
\]
Similarly
\[
\left\Vert P\left(  x\right)  \cdot\operatorname{Im}\left(  \left\langle
x ,c\right\rangle_2 ^{m}\right)  \right\Vert _{a}
\leq
\left\Vert P\left(
x\right)  \left\langle x ,c\right\rangle_2 ^{m}\right\Vert _{a}
\leq
C^{1/2}\left\Vert \left\langle x,c\right\rangle_2 ^{m}\right\Vert _{a}\cdot
m^{(k-\rho-1)/2}.
\]
If $\left\Vert \operatorname{Re} (\left\langle x,c\right\rangle_2) \right\Vert
_{a}\geq\frac{1}{2}\left\Vert \left\langle x,c\right\rangle_2 \right\Vert _{a}$, we choose  
$q_{m} := \operatorname{Re} (\left\langle x,c\right\rangle_2)$. Otherwise we let $q_{m} := \operatorname{Im} (\left\langle x,c\right\rangle_2)$. In either case 
  there exists a $q_{m}$ with real coefficients such that
\[
\left\Vert P\left(  x\right)  \cdot q_{m}\right\Vert _{a}^{2}\leq 4 C\left\Vert
q_{m}\right\Vert _{a}^{2}\cdot m^{k-\rho-1}.
\]
\end{proof}

\end{document}